\documentclass[11pt]{article} 
\usepackage{graphicx, amssymb, latexsym, float, xcolor, amsmath, tikz}
\usepackage{pdfpages}
\allowdisplaybreaks

\newtheorem{defin}{}
\newtheorem{saetze}[defin]{}
\newtheorem{conjec}[defin]{}
\newtheorem{lemmas}[defin]{}
\newtheorem{folger}[defin]{}
\newtheorem{bemerk}[defin]{}
\newtheorem{prop}[defin]{}

\newenvironment{theorem}  {\begin{saetze}\it {\bf Theorem:}}{\end{saetze}}
\newenvironment{definition}{\begin{defin}\it {\bf Definition:}}{\end{defin}}

\newenvironment{lemma}    {\begin{lemmas}\it {\bf Lemma:}}{\end{lemmas}}
\newenvironment{corollary}{\begin{folger}\it {\bf Corollary:}}{\end{folger}}
\newenvironment{remark}   {\begin{bemerk}\rm {\it Remark:}}{\end{bemerk}}
\newenvironment{proof}    {{\it Proof}:}{{\hfill \fillbox \bigskip}}

\newcommand{\fillbox}{\mbox{$\bullet$}}
\newcommand{\ra}{\rightarrow}

\newcommand{\ms}{\mapsto}
\newcommand{\ol}{\overline}

\newcommand{\hH}{\hat{H}}
\newcommand{\N}{\mathbb N}

\newcommand{\Z}{\mathbb Z}
\newcommand{\Q}{\mathbb Q}

\newcommand{\G}{\mathcal G}
\newcommand{\B}{\mathcal B}
\newcommand{\C}{\mathcal C}
\newcommand{\T}{\mathcal T}

\newcommand{\LL}{\mathcal L}

\renewcommand{\O}{\mathcal O}
\newcommand{\U}{\mathcal U}
\newcommand{\p}{\mathfrak p}

\newcommand{\Hom}{\mathrm{Hom}}

\newcommand{\ttt}{\theta}
\newcommand{\kkk}{\kappa}

\newcommand\blfootnote[1]{%
  \begingroup
  \renewcommand\thefootnote{}\footnote{#1}%
  \addtocounter{footnote}{-1}%
  \endgroup }

\newenvironment{items}{\begin{list}{$\alph{item})$}
{\labelwidth18pt \leftmargin18pt \topsep3pt \itemsep1pt \parsep0pt}}
{\end{list}}

\begin{document}

\title{Lie rings related to the \\ $p$-groups of maximal class}
\author{Bettina Eick, Patali Komma and Subhrajyoti Saha}
\date{\today}
\maketitle

\begin{abstract}
The Lazard correspondence induces a close relation between the $p$-groups 
of maximal class and a certain type of Lie ring constructed from $p$-adic 
number fields. Our aim here is to investigate such Lie rings. In particular, 
we show that they are always finite. It then follows that they are nilpotent
of small class. These results close an important gap in (Eick, Komma \& Saha
2025). \blfootnote{{\bf Acknowledgement:} The second author was supported 
by an Alexander von Humboldt Foundation Research Fellowship.
The third author acknowledges the support of the Royal Society through 
Research Grant RGS \textbackslash R2 \textbackslash 252076}

\end{abstract}

%%%%%%%%%%%%%%%%%%%%%%%%%%%%%%%%%%%%%%%%%%%%%%%%%%%%%%%%%%%%%%%%%%%%%%%%%%%%%
\section{Introduction}

The classification of $p$-groups of maximal class is a long-standing 
project in group theory. It has many interesting and deep results so far 
and also many open problems still left. For small primes $p \in \{2,3\}$,
a complete classification was achieved by Blackburn \cite{Blackburn1958}. 
Hence we focus on the case $p \geq 5$ here. We refer to Leedham-Green \& 
McKay \cite{LGM2002} for an introduction to the state of the art, references 
and background.

The $p$-groups of maximal class can be visualized via their associated graph 
$\G(p)$: its vertices correspond one-to-one to the infinitely many 
isomorphism types of $p$-groups of maximal class and there is an edge 
$G \rightarrow H$ if $H / Z(H) \cong G$ holds. It is known that $\G(p)$
consists of an isolated point and an infinite tree $\T$ having a unique
infinite path, called its mainline. We denote this infinite path by $S_2 
\ra S_3 \ra \ldots$ with $S_i$ being a group of order $p^i$. The $i$th 
branch $\B_i$ of $\T$ is the subtree consisting of all descendants of $S_i$ 
that are not descendants of $S_{i+1}$; It is a finite tree. The graph 
$\G(p)$ consists of its isolated
point, its mainline and its sequence of branches $\B_2, \B_3, \ldots$.
Understanding the branches is the main problem in the classification
of $p$-groups of maximal class.

Let $G$ be a group in a branch $\B_i$ for $i \geq p-1$ and let $P(G)$ be its
two-step centralizer; that is, the centralizer in $G$ of the lower central
series quotient $\gamma_2(G)/\gamma_4(G)$. Then $P(G)$ is a maximal subgroup
of $G$. By Shepard \cite{Shepherd1970}, the subgroup $P(G)$ has class at most 
$(p-1)/2$. Thus the Lazard correspondence applies to $P(G)$ yielding a Lie 
ring $L(P(G))$. A central aim is to understand these Lie rings.

In \cite{EKS25} it is shown that, if $G$ is not a leaf in its branch 
$\B_i$, then its Lie ring $L(P(G))$ arises from a construction based on 
$p$-adic number theory as we outline in the following.

\subsection{Construction of Lie $p$-rings}
\label{liering}

Let $K = \Q_p(\ttt)$ where $\Q_p$ denote the field of $p$-adic rationals
and $\ttt$ is a primitive $p$-th root of unity. Let $\O$ be the maximal order
in $K$ and note that $\O$ has a unique maximal ideal $\p$ generated by 
$\kkk = \ttt-1$. The ring $\O$ has a unique chain of ideals $\O = \p^0 
> \p^1 > \p^2 > \ldots > \{0\}$, with $\p^i$ generated by $\kkk^i$ and
has index $p^i$ in $\O$. We consider $\p^i \wedge \p^i$ as $\ttt$-module 
under diagonal action and define
\[ H_i = \Hom_\ttt(\p^i \wedge \p^i, \p^{2i+1}) \;\;\; \mbox{ and } \;\;\;
 \hH_i = \{ \gamma \in H_i \mid \gamma \mbox{ surjective }\}.\]

For $\gamma \in \hH_i$ and $u,v,w \in \p^i$ write
\[J_{\gamma}(u, v, w)=\gamma(\gamma(u\wedge v)\wedge w)+\gamma(\gamma(v\wedge w)\wedge u)+\gamma(\gamma(w\wedge u)\wedge v).\]
Define $J(\gamma)$ as the ideal in $\O$ generated by the set
$\{ J_\gamma(u,v,w) \mid u,v,w \in \p^i \}$. Then there exists 
$\lambda(\gamma) \in \N_0 \cup \{ \infty \}$ so that
\[ J(\gamma)=\p^{\lambda(\gamma)};\]
where $\p^\infty = \{0\}$.  We define $L_{i,m}(\gamma)$ 
as the set $\p^i / \p^m$ with natural addition and multiplication 
$[x+\p^m, y+\p^m] = \gamma(x\wedge y)+ \p^m$. If $m \leq \lambda(\gamma)$, then 
$L_{i,m}(\gamma)$ is a Lie ring, and it has order $p^{m-i}$. A central 
result in \cite{EKS25} states the following.

\begin{theorem} {\bf (EKS \cite{EKS25})} \\
\label{EKSI}
Let $p \geq 5$, let $i >  p+1$ and let $G$ a group in $\B_i$ so that $G$
is not a leaf. Write $P$ for the two-step centralizer in $G$ and $m = 
\log_p|P| + i$. Then there exists $\gamma \in \hH_i$, so that 
\[ L(P) \cong L_{i,m}(\gamma).\]
\end{theorem}

\subsection{Investigation of Lie $p$-rings}

Our aim is now to understand the Lie rings $L_{i,m}(\gamma)$ as arising
in Theorem \ref{EKSI}. These satisfy $m \leq \lambda(\gamma)$ and are
quotients of $L_{i,\lambda(\gamma)}(\gamma)$. A first step is now to 
investigate $\lambda(\gamma)$. Note that if $L_{i,\lambda(\gamma)}(\gamma)$ 
is finite, then it has size $p^{\lambda(\gamma)-i}$ by construction. Our 
central result here is the following, see Section \ref{pmt2} for a proof 
including explicit bounds on $y$. 

\begin{theorem} {\bf (Main Theorem I)} \\
\label{mt2}
Let $p\ge 5$ and $i \in \N_0$.
If $\gamma\in \hat{H}_i$ and $\lambda = \lambda(\gamma)$, then 
$L_{i,\lambda}(\gamma)$ is finite; more precisely, there exists
$y \in \N_0$ with $\lambda = 3i + 13 - 2p + y$ or, equivalently,
there exists $y \in \N_0$ with
\[ |L_{i,\lambda}(\gamma)| = p^{2i+13-2p+y}.\]
\end{theorem}

By \cite[Th. 1]{EKS25},  the finiteness of $L_{i,\lambda}(\gamma)$
implies that $L_{i,\lambda}(\gamma)$ is nilpotent and, if $i > p-1$, then
the class of $L_{i,\lambda}(\gamma)$ is at most $p-1$. In the later case,
the Lazard correspondence applies defining a group $G(L)$ for each of the
arising Lie rings $L$. This induces the following result.

\begin{corollary} 
Let $p \geq 5$ and $i > p+1$. Write $\LL_i = \{ L_{i,m}(\gamma) \mid 
\gamma \in \hH_i, m \leq \lambda(\gamma) \}$ and let $\C_i$ the set of 
non-leafs in $\B_i$. Then the Lazard correspondence induces a surjection 
\[ \LL_i \ra \C_i : L \ra G(L) \rtimes  \langle \ttt \rangle, \]
where $\ttt$ acts by multiplication on $\O$ and thus on $L$ and on $G(L)$.
\end{corollary}

\subsection{Special cases}

For $j \in \Z$ coprime to $p$ let $\sigma_j$ be the Galois automorphism 
of $K$ mapping $\ttt$ to $\ttt^j$. For $a \in\{2, \ldots,(p-1) / 2\}$ 
define $\vartheta_a: K \wedge K \rightarrow K$ via 
\[ \vartheta_a(x \wedge y)=\sigma_a(x) \sigma_{1-a}(y)
                              -\sigma_{1-a}(x) \sigma_a(y) .\]

As observed in \cite[Lemma 6]{EKS25}, $\vartheta_a \in \hH_i$ for each 
$i \in \N_0$ and the homomorphisms $\vartheta_a$ span $\hH_i$ over $K$. 
If $p=5$, then it is sufficient to choose $\gamma = c \vartheta_2$ for
some $c \in \U$, the unit group of $\O$. We thus obtain the following
special cases. 

\begin{theorem} {\bf (Main Theorem II)} \\
\label{mt3} 
Let $p\ge 5$ and $i \in \N_0$. Let $\gamma = c \vartheta_a$ for some 
$c \in \U$ and $2 \leq a \leq (p-1)/2$ and write $\lambda = \lambda(
\gamma)$. Then $\gamma \in \hH_i$ and
$|L_{i,\lambda}(\gamma)| = p^{2i+3+y}$ for some $y \in \{0, 1, 2\}$. 
If $p=5$, then $y = 0$ holds.
\end{theorem}

GAP experiments suggest that $y \in \{0,2\}$ if $\gamma = \vartheta_a$.
While the case $y=2$ does not arise for $p=5$, it does arise for larger
primes. For example:
\begin{items}
\item[$\bullet$]
if $p=7$, then $y=2$ if $(a,i) \in \{(2,3), (3,2), (3, 5)\}$.
\item[$\bullet$]
if $p=11$, then $y=2$ if $(a, i) \in \{(4,10)\}$.
\item[$\bullet$]
if $p=13$, then $y=2$ if $(a, i) \in \{(2,4), (4,2), (4,5),(4,8),
 (4,11), (5,4), (6,3), (6,9)\}$.
\end{items}

%%%%%%%%%%%%%%%%%%%%%%%%%%%%%%%%%%%%%%%%%%%%%%%%%%%%%%%%%%%%%%%%%%%%%%%%%%%%%
\section{Preliminaries on number theory}

This section recalls some preliminaries on number theory.
As before, let $K = \Q_p(\ttt)$, where $\ttt$ is a primitive $p$-root of 
unity, let $\O$ be the maximal order in $K$ and let $\p$ denote the unique
maximal ideal in $\O$. Then $\kkk$ generates $\p$ as ideal and $\U = \O 
\setminus \p$ are the units in $\O$.

\begin{definition}
For $v \in K$ we define the valuation of $v$ as $val(v) = n$ if $v \in \p^n \setminus \p^{n+1}$
and we define $val(0) = \infty$. For a vector $v$ over $K$ we define 
$val(v)$ as the minimum of the valuations of the entries of $v$.
\end{definition}

The ideal $\p^i$ of $\O$ has the natural $\Z_p$-basis $\{ \ttt^j \kkk^i \mid
0 \leq j < p-1\}$. In applications, it is often useful to consider a
different basis as in the following remark.

\begin{remark}
\label{gens}
Let $i \in \N_0$ and $d=p-1$ the $\Z_p$-dimension of $\O$.
\begin{items}
\item[\rm (a)]
The set $\{ \kkk^j \mid i \leq j < i+d \}$ is a $\Z_p$-basis for $\p^i$.
\item[\rm (b)]
The set $\{ \kkk^j \wedge \kkk^k \mid i \leq j < k < i+d\}$ generates 
$\p^i \wedge \p^i$ as $\Z_p$-module.
\item[\rm (c)]
$\kappa^d \equiv - p \bmod \p^{d+1}$.
\end{items}
\end{remark}

\begin{proof}
(a) and (b) are well-known and it remains to prove (c). First note that
$\theta^r = (1+\kappa)^{r}=1+r\kappa+{{r}\choose{2}}\kappa^2+\cdots 
+{{r}\choose{r}}\kappa^r$ for $1 \leq r \leq p-1$ and 
$\sum_{r=0}^{p-1}{{p-1}\choose{r}}={{p}\choose{r+1}}$. 
As $p \in \p^p$, it follows that
\[
0 = 1+\theta+\cdots+\theta^{p-1}
   = p + \binom{p}{2}\kappa + \cdots + \binom{p}{p}\kappa^{p-1}
   \equiv p + \kappa^{p-1} \bmod{\p^{p}}.
\]
As $d = p-1$, this yields the result of (c).
\end{proof}

Let $\omega\in \Q_p$ be a primitive $(p-1)$-th root of unity 
and $r$ be the generator of the multiplicative group of integers 
$(\Z/ p \Z)^*$ such that $\omega-r\in \p$. Then let 
$\sigma=\sigma_r$ and note that $\sigma$ generates $Gal(K)$.
The following is proved in \cite[Section 2]{LGM1984}.

\begin{theorem}
\begin{items}
\item[\rm (a)] 
The eigenvalues of $\sigma$ are $1, \omega, \ldots, \omega^{p-2}$, with 
each eigenspace of dimension $1$.
\item[\rm (b)]
If $v\in K$ is an eigenvector of $\sigma$, then the corresponding 
eigenvalue is $\omega^{val(v)}$. 
\item[\rm (c)]
For each $n\in \Z$ there exists an eigenvector $v\in \p^n\setminus
\p^{n+1}$ of $\sigma$.
\end{items}
\end{theorem}

The following result is an easy consequence.

\begin{corollary}
\label{R2}
Let $v$ be an eigenvector of $\sigma$ with $val(v)=n$. Then $v$ is an 
eigenvector of $\sigma^m$ with eigenvalue $\omega^{mn}$. Further if 
$a\in \{2, \ldots (p-1)/2\}$ with $\sigma_a=\sigma^k$ and 
$\sigma_{1-a}=\sigma^l$ for some integers $k$ and $l$, then
\begin{align*}
\sigma_a(v)=\omega^{kn}v, \quad
\sigma_{1-a}(v)=\omega^{ln}v, \quad 
\sigma_{a^2}(v)=\omega^{2kn}v, \quad \\
\sigma_{(1-a)^2}=\omega^{2ln}\omega, \quad 
\sigma_{a(1-a)}(v)=\omega^{(k+l)n}v. 
\end{align*}
\end{corollary}

\begin{proof}
Let $v$ be an eigenvector of $\sigma$ with $val(v)=n$. Then 
$\sigma(v)=\omega^nv$, and $\sigma^2(v)=\sigma(\omega^nv)$. 
As $\omega\in \Q_p$ is fixed by $\sigma$, it follows that
$\sigma^2(v)=\omega^n\sigma(v)=\omega^{2n}v$. Iterating this idea,
it follows that $v$ is an eigenvector of $\sigma^m$ with eigenvalue 
$\omega^{mn}$. Using the definitions of $l$ and $k$, the result
follows.
\end{proof}

\begin{lemma}
\label{unitfixedgalois}
The group $Gal(K)$ acts trivially on $\O / \p$.
\end{lemma}

\begin{proof}
Each element $c\in \O$ can be written as $c = \sum_{r=0}^m a_r \kappa^r$ 
for some $m\in \N_0$ and $a_r\in \{0, \ldots, p-1\}$. Thus $c \equiv a_0
\bmod \p$ and $\sigma(c) = \sum_{r=0}^m a_r \sigma(\kappa)^r \equiv a_0 
\bmod \p$. 
\end{proof}

%%%%%%%%%%%%%%%%%%%%%%%%%%%%%%%%%%%%%%%%%%%%%%%%%%%%%%%%%%%%%%%%%%%%%%%%%%%%%
\section{Preliminaries on homomorphism spaces}
\label{dandx}

Throughout, we assume that $p \geq 5$ is prime and $i \in \N_0$.  Our aim in
this section is to provide background on the space 
\[ \hH_i \subseteq H_i = \Hom_\ttt( \p^i \wedge \p^i, \p^{2i+1}).\]

The structure of $\Hom_\ttt(\O \wedge \O, \O)$ was analysed by Leedham-Green
\& McKay, see \cite{LGM2002}. Based on this, Dietrich \& Eick \cite{DE2017} 
proved that
every element of $\Hom_\ttt(\O \wedge \O, \O)$ can be written uniquely 
as $\sum_{a=2}^{(p-1)/2} c_a \vartheta_a$ with coefficients $c_a \in 
\p^{4-p}$. The following remark translates this to $\hH_i$.

\begin{remark}
\begin{items}
\item[\rm (a)]
$\gamma \in \Hom_\ttt( \p^i \wedge \p^i, \p^{2i+1})$ induces 
$\hat{\gamma} \in \Hom_\ttt( \O \wedge \O, \O)$ via 
\[ \hat{\gamma} : \O \wedge \O \ra \O 
: (x \wedge y) \ms \kkk^{-(2i+1)} \gamma( \kkk^i x \wedge \kkk^i y).\]
\item[\rm (b)]
$\delta \in \Hom_\ttt(\O \wedge \O, \O)$ induces 
$\ol{\delta} \in \Hom_\ttt( \p^i \wedge \p^i, \p^{2i+1})$ via
\[ \ol{\delta} : \p^i \wedge \p^i \ra \p^{2i+1}
: (x \wedge y) \ms \kkk^{(2i+1)} \delta( \kkk^{-i} x \wedge \kkk^{-i} y).\]
\item[\rm (c)]
The two maps are inverse to each other and hence induce isomorphisms between
$\Hom_\ttt(\O \wedge \O, \O)$ and $\Hom_\ttt(\p^i \wedge \p^i, \p^{2i+1})$.
\end{items}
\end{remark}

\begin{lemma}
\label{low}
Each $\gamma \in \hH_i$ can be written uniquely in the form 
$\gamma = \sum_{a=2}^{(p-1)/2}
c_a \vartheta_a$ with $c_a \in K$ and $val(c_a) \geq 5-p$.
\end{lemma}

\begin{proof}
Let $\gamma \in \hH_i$. Then $\hat{\gamma} = \sum c_a \vartheta_a$ and
$\gamma = \ol{\hat{\gamma}}$. As $\vartheta_a( \kkk^{-i} x \wedge
\kkk^{-i} y) = \kkk^{-2i} u_{a,i} \vartheta_a( x \wedge y )$ for some
unit $u_{a,i}$, it follows that $\gamma = \sum_a \kkk c_a u_{a,i} 
\vartheta_a$ and the valuation of the coefficients in this sum is 
given by $val(\kkk c_a u_{a,i}) = val( c_a ) + 1 \geq 5-p$.
\end{proof}

For $\gamma \in \hH_i$ we write $val( \gamma ) = \min \{ val(c_a) 
\mid 2 \leq a \leq (p-1)/2\}$. Lemma \ref{low} has the following 
corollary.

\begin{corollary}
\label{mind}
Let $\gamma \in \hH_i$ with $v = val(\gamma)$.
\begin{items}
\item[\rm (a)]
$0 \geq v \geq 5-p$.
\item[\rm (b)]
$\gamma( \p^j \wedge \p^k ) \in \p^{j+k+v}$ for each $j,k \geq i$.
\end{items}
\end{corollary}

\begin{proof}
Write $\gamma = \sum c_a \vartheta_a$. \\
(a) This follows from Lemma \ref{low}. Note that not all 
coefficients $c_a$ can be in $\p$. \\
(b) Note that $\vartheta_a( \p^j \wedge \p^k) \leq \p^{j+k}$ for all $j,k$. 
Thus $c_a \vartheta_a( \p^j \wedge \p^k) \in \p^{j+k+v}$ for all $a, j, k$. 
\end{proof}

%%%%%%%%%%%%%%%%%%%%%%%%%%%%%%%%%%%%%%%%%%%%%%%%%%%%%%%%%%%%%%%%%%%%%%%%%%%%%
\section{A lower bound}

Throughout this section we assume that $p \geq 5$ is prime and $i \in 
\N_0$. Let $\gamma \in \hH_i$ and write $\gamma = \sum c_a \vartheta_a$
and $\lambda = \lambda(\gamma)$. 
Recall that $J( \gamma ) = \p^\lambda$. This section introduces a 
lower bound for $\lambda$. 

For each $x \in \Z$, let $e_x$ be an eigenvector of $\sigma$ with 
$val(e_x) = x$. Then $\sigma(e_x) = \omega^x e_x$. Recall that $\omega
\in \Z_p$ is a primitive $(p-1)$-th root of unity with $\omega \equiv
r \mod p$. Let $\sigma_a = \sigma^{k_a}$ and $\sigma_{1-a} = \sigma^{l_a}$
for certain $k_a, l_a \in \N$. 
The next three results follow from direct computations. 

\begin{lemma}
For $u,v \in \Z$ define $m_a(u,v) = \omega^{u k_a + v l_a} \in \Z_p$. 
For $c \in K$ it follows that
\[ \vartheta_a(c e_x \wedge e_y) =
    (\sigma_a(c) m_a(x,y) - \sigma_{1-a}(c) m_a(y,x)) e_{x}e_y.\]
\end{lemma}

\begin{lemma}
For $x,y \in \Z$ define $t_a(x,y) = m_a(x,y)-m_a(y,x) \in \Z_p$. Then 
\begin{eqnarray*}
&& \gamma( \gamma( e_x \wedge e_y ) \wedge e_z)  \\
&=& (\sum_{a,b} c_a \sigma_a(c_b) t_b(x,y) m_a(x+y,z)
             - c_a \sigma_{1-a}(c_b) t_b(x,y) m_a(z,x+y)) e_{x}e_ye_z.
\end{eqnarray*}
\end{lemma}

Define $l_{a,b}(x,y,z), r_{a,b}(x,y,z) \in \Z_p$ and $\Gamma_\gamma(x,y,z)
\in K$ as
\begin{eqnarray*}
l_{a,b}(x,y,z) &=& 
     t_b(x,y) m_a(x+y,z) 
   + t_b(y,z) m_a(y+z,x)
   + t_b(z,x) m_a(z+x,y), \\
r_{a,b}(x,y,z) &=& 
     t_b(x,y) m_a(z,x+y) 
   + t_b(y,z) m_a(x,y+z)
   + t_b(z,x) m_a(y,z+x), \\
\Gamma_\gamma(x,y,z) &=&
\sum_{a,b} c_a \sigma_a(c_b) l_{a,b}(x,y,z) 
    - c_a \sigma_{1-a}(c_b) r_{a,b}(x,y,z).
\end{eqnarray*}

\begin{lemma}
\label{nestedtgamma}
For each $x,y,z \in \N_0$ it follows that
\begin{eqnarray*}
&& \gamma(\gamma(e_x \wedge e_y) \wedge e_z)
+ \gamma(\gamma(e_y \wedge e_z) \wedge e_x)
+ \gamma(\gamma(e_z \wedge e_x) \wedge e_y) \\
&=& \Gamma_\gamma(x,y,z) e_{x}e_ye_z.
\end{eqnarray*}
\end{lemma}

The following theorem proves a first step towards Theorems \ref{mt2}
and \ref{mt3}. 

\begin{theorem}
\label{lboundlambda}
Let $\gamma \in \hH_i$ with $v = val(\gamma)$ and write $\lambda = 
\lambda(\gamma)$. Then $\lambda \geq 3i+ 3+2v$.
Thus $\lambda \geq 3i+13-2p$ in general and $\lambda \geq 3i+3$ if 
$\gamma = c \vartheta_a$. 
\end{theorem}

\begin{proof}
Recall that $J(\gamma) = \p^\lambda$. By Lemma \ref{nestedtgamma}, it follows
that $J(\gamma)$ is generated by $\{ \Gamma_\gamma(x,y,z) e_{x}e_ye_z \mid 
x>y>z \geq i\}$ as ideal. Thus $\lambda = \min \{ val( \Gamma_\gamma(x,y,z))
+x+y+z \mid x>y>z \geq i\}$. The term $\Gamma_\gamma(x,y,z)$ is combined of
elements in $\Z_p$ and coefficients $c_a$ and their images under the Galois
group. Recall that $5-p \leq v \leq 0$ by Corollary \ref{mind}. Thus
all of these elements are in $\p^{2v}$ and hence 
$\Gamma_\gamma(x,y,z) \geq 2 v$ for all $x,y,z$. The minimal 
choice for $x,y,z$ is $i+i+1+i+2 = 3i+3$. This yields the desired result.
\end{proof}

%%%%%%%%%%%%%%%%%%%%%%%%%%%%%%%%%%%%%%%%%%%%%%%%%%%%%%%%%%%%%%%%%%%%%%%%%%%%%
\section{Upper bounds}

Throughout this section we assume that $p \geq 5$ is prime and 
$i \in \N_0$. We consider $\gamma \in \hH_i$ and write $\gamma = \sum c_a 
\vartheta_a$ and $\lambda = \lambda(\gamma)$. Recall that $J(\gamma) = 
\p^\lambda$. This section 
shows that $\lambda$ is finite and it introduces an upper bound for it.

%%%%%%%%%%%%%%%%%%%%%%%%%%%%%%%%%%%%%%%%%%%%%%%%%%%%%%%%%%%%%%%%%%%%%%%%%%%%%
\subsection{The offset of a homomorphism}

Our first aim is to analyse the images $\gamma( \p^j \wedge \p^k )$ 
for all $j, k \geq i$. We define $\rho(j,k) \in \Z$ via 
\[ \gamma( \p^j \wedge \p^k ) = \p^{j+k+\rho(j,k)}.\]
Recall that $d = p-1$ is the $\Z_p$-dimension of each $\p^j$ and $\p^{j+d} 
= p \p^j$ holds. Thus for each $j, k \geq i$ it follows that
\[ \rho(j,k) = \rho(j+d,k) = \rho(j,k+d). \]
Hence there are only finitely many different values in the set
$\{ \rho(j, k) \mid j, k \geq i \}$; This is used in the first 
part of the following definition. Further, recall that every element 
$x \in K$ of valuation $v$ can be written (uniquely) as 
$x = \sum_{j \geq v} a_j \kkk^j$ with $a_j \in \{0, \ldots, p-1\}$.

\begin{definition}
Let $\gamma\in \hat{H}_i$. 
\begin{items}
\item[\rm (a)]
The {\em offset} $\rho(\gamma)$ of $\gamma$ is 
$\rho(\gamma) = \min \{ \rho(j,k) \mid j, k \geq i \}$.
\item[\rm (b)]
For $j, k \geq i$ let $a(j,k) \in \{0, \ldots, p-1\}$ defined by
\[ \gamma(\kkk^j \wedge \kkk^k) \equiv a(j,k) \kkk^{j+k+\rho(\gamma)} \bmod 
\p^{j+k+\rho(\gamma)+1}.\]
\end{items}
\end{definition}

The offset is an integer. An lower bound for it is given in the 
following lemma.

\begin{lemma}
Let $\gamma\in \hat{H}_i$. Then $5-p \leq \rho(\gamma)$.
\end{lemma}

\begin{proof}
By Lemma \ref{low} the homomorphism $\gamma$ can be written as 
$\gamma = \sum c_a \vartheta_a$ with $val(c_a) \geq 5-p$. By 
\cite[Lemma 6]{EKS25} we note that $\vartheta_a( \p^j \wedge \p^k )
= \p^{j+k+\epsilon}$ with $\epsilon \in \{0,1\}$. Thus 
$\gamma( \p^j \wedge \p^k ) \subseteq \p^{j+k+5-p}$ and hence
$\rho(j,k) \geq 5-p$ for each $j,k$. 
\end{proof}

As examples, we exhibit the arrays of $a(j,k)$ for $p=7$ and the two
homomorphisms $\vartheta_2$ (on the left) and $\vartheta_3$ (on the 
right) in the range $0 \leq j,k < 6$ with $i=0$.
\[
\left( \begin{array}{ccccccc}
 0& 4& 4& 5& 6& 2 \\ 
 3& 0& 6& 6& 4& 2 \\
 3& 1& 0& 2& 2& 6 \\
 2& 1& 5& 0& 3& 3 \\
 1& 3& 5& 4& 0& 1 \\
 5& 5& 1& 4& 6& 0 \\ 
\end{array} \right ) 
\;\;\;\;\;\;\;
\left( \begin{array}{ccccccc}
  0& 2& 2& 0& 5& 5 \\ 
  5& 0& 2& 2& 0& 5 \\
  5& 5& 0& 2& 2& 0 \\
  0& 5& 5& 0& 2& 2 \\
  2& 0& 5& 5& 0& 2 \\
  2& 2& 0& 5& 5& 0 \\
\end{array} \right )
\]

A main part of our proof for the upper bound on $\lambda$ consists in
understanding the coefficients $a(j,k)$ as good as possible. A first step
towards this aim is the following lemma.
 
\begin{lemma}
\label{lem9}
Let $\gamma\in \hat{H}_i$, let $\rho = \rho(\gamma)$ and let $d = p-1$. 
\begin{items}
\item[\rm (a)] 
$a(j,j) = 0$ for all $j \geq i$.
\item[\rm (b)]
$a(j,k) \equiv -a(k,j) \bmod p$ for all $j,k \geq i$.
\item[\rm (c)]
$a(j,k) = a(j+d,k) = a(j,k+d)$ for all $j,k \geq i$.
\item[\rm (d)] 
$a(j, k) \equiv a(j+1, k)+a(j, k+1) \bmod p$ for all $j,k \geq i$.
\item[\rm (e)]
$a(j,j+1) = a(j,j+2)$ for all $j \geq i$.
\item[\rm (f)]
$a(j, j+d-2) = a(j+d-1, j+d-2)$ for all $j \geq i$.
\end{items}
\end{lemma}

\begin{proof}
Throughout the proof let $j,k \geq i$. \\
(a) This follows directly, since $\gamma( \kkk^j \wedge \kkk^j ) = 0$. \\
(b) This is obtained by $\gamma( \kkk^j \wedge \kkk^k ) = - \gamma(
\kkk^k \wedge \kkk^j)$. \\
(c) First note that $\gamma( \kkk^{j+d} \wedge \kkk^k ) 
= \gamma( \kkk^d \kkk^j \wedge \kkk^k)$. Next, $\kkk^d \equiv -p \bmod
\p^{d+1}$, see Remark \ref{gens}. Thus calculating modulo 
$\p^{j+k+d+\rho+1}$, we obtain the following:
\begin{eqnarray*}
a(j+d, k) \kkk^{j+d+k+\rho} 
 &\equiv& \gamma(\kkk^{j+d} \wedge \kkk^k) \\
 &\equiv& \gamma( -p \kkk^j \wedge \kkk^k ) \\
 &=& -p \gamma( \kkk^j \wedge \kkk^k )  \\
 &\equiv& -p a(j,k) \kkk^{j+k+\rho} \\
 &\equiv& a(j,k) \kkk^d \kkk^{j+k+\rho} \\
 &=& a(j,k) \kkk^{j+k+d+\rho}. 
\end{eqnarray*}
Hence $a(j+d,k) = a(j,k)$ follows. Similarly, $a(j,k)= a(j,k+d)$ holds. \\
(d) 
Since $\gamma$ is bilinear and $\theta=\kappa+1$, it follows for $\alpha,
\beta \in \p^i$ that
\begin{eqnarray*}
\kappa\gamma(\alpha\wedge \beta)
 &=& \gamma(\theta\alpha\wedge \theta\beta)-\gamma(\alpha\wedge \beta) \\
 &=& \gamma(\kappa\alpha\wedge \kappa\beta)
     +\gamma(\kappa\alpha\wedge \beta)
     +\gamma(\alpha\wedge \kappa\beta).
\end{eqnarray*}
We apply this with $(\alpha, \beta)=(\kappa^j, \kappa^k)$. Using that
$\gamma( \kkk^{j+1} \wedge \kkk^{k+1} ) \in \p^{j+k+\rho+2}$, we obtain
the following modulo $\p^{j+k+\rho+2}$:
\begin{eqnarray*}
a(j,k) \kkk^{j+k+\rho+1} 
 &\equiv& \kappa \gamma(\kappa^j\wedge \kappa^k) \\
 &\equiv& \gamma(\kappa^{j+1}\wedge \kappa^k)
           +\gamma(\kappa^j\wedge \kappa^{k+1}) \\
 &\equiv& (a(j+1,k) + a(j,k+1)) \kkk^{j+k+\rho+1}.
\end{eqnarray*}
This yields the desired result. \\
(e) We combine the previous parts to $a(j,j+1) = a(j+1,j+1) + a(j, j+2)
= a(j,j+2)$. \\
(f) Again, we combine the previous parts and obtain that 
$a(j+d-1, j+d-2) = a(j+d, j+d-2) + a(j+d-1, j+d-1) = a(j+d, j+d-2) =
a(j, j+d-2)$. 
\end{proof}

%%%%%%%%%%%%%%%%%%%%%%%%%%%%%%%%%%%%%%%%%%%%%%%%%%%%%%%%%%%%%%%%%%%%%%%%%%%%%
\subsection{The Jacobi elements}

Let $\gamma \in \hH_i$. Recall that $\p^\lambda = J(\gamma) = 
\langle J_\gamma(u,v,w) \mid u,v,w \in \{\kkk^j \mid j \geq i\} \rangle$, 
where 
\[ J_\gamma(u,v,w) 
  = \gamma( \gamma( u \wedge v ) \wedge w)
  + \gamma( \gamma( v \wedge w ) \wedge u)
  + \gamma( \gamma( w \wedge u ) \wedge v).\]
We define for $j, k, l \geq i$ the integers $J(j,k,l) \in \{0, \ldots, 
p-1\}$ via
\[ J(j, k, l) := a(j, k) a(j+k+\rho,  l)
                +a(k, l) a(k+l+\rho,  j)
                +a(l, j) a(l+j+\rho,  k) \bmod p.\]

\begin{lemma}
Let $\gamma \in \hH_i$ with $\rho = \rho(\gamma)$ and let $j,k,l \geq i$.
Then 
\[ J_\gamma( \kkk^j, \kkk^k, \kkk^l) \equiv J(j,k,l) \kkk^{j+k+l+2 \rho} 
\bmod \p^{j+k+l+2 \rho+1}.\]
\end{lemma}

\begin{proof}
We evaluate modulo $\p^{j+k+l+ 2\rho+1}$:
\begin{eqnarray*}
\gamma( \gamma( \kkk^j \wedge \kkk^k ) \wedge \kkk^l )
&\equiv& \gamma( a(j,k) \kkk^{j+k+\rho} \wedge \kkk^l ) \\
&\equiv& a(j,k) \gamma( \kkk^{j+k+\rho} \wedge \kkk^l ) \\
&\equiv& a(j,k) a(j+k+\rho, l) \kkk^{j+k+l+2\rho} 
\end{eqnarray*}
and this yields the desired result.
\end{proof}

As examples, we exhibit the arrays of $J(j,j+1,l)$ for $p=7$ and the two
homomorphisms $\vartheta_2$ (on the left) and $\vartheta_3$ (on the 
right) in the range $0 \leq j,l < 6$ with $i=0$.

\[
\left( \begin{array}{cccccc}
   0& 0& 4& 4& 4& 5 \\ 
   4& 0& 0& 1& 1& 3  \\
   0& 1& 0& 0& 6& 6  \\
   0& 5& 6& 0& 0& 0  \\
   3& 3& 6& 0& 0& 0  \\
   0& 5& 5& 2& 3& 0 \\
\end{array} \right) 
\;\;\;\;\;\;\;
\left( \begin{array}{cccccc}
   0& 0& 1& 1& 0& 0  \\
   1& 0& 0& 1& 1& 1  \\
   1& 1& 0& 0& 0& 0  \\
   6& 0& 0& 0& 0& 6  \\
   6& 6& 6& 6& 0& 0  \\
   0& 0& 0& 6& 6& 0 \\
\end{array} \right) 
\]

The Jacobi elements $J(j,k,l)$ will be our main tool in proving that
$\lambda$ is finite and in determining an upper bound for it. The 
following theorem will be a central key towards this.

\begin{theorem}
\label{bound}
Let $\gamma \in \hH_i$ and write $\rho = \rho(\gamma)$ and $d = p-1$. If 
there exist $j,k,l \in \{i, \ldots, i+d-1\}$ with $J(j,k,l) \neq 0$, then 
$\lambda = \lambda(\gamma)$ is finite with $\lambda \leq 3i+3d-6+2\rho$.
\end{theorem}

\begin{proof}
If $J(j,k,l) \neq 0$, then $J_\gamma( \kkk^j, \kkk^k, \kkk^l) \neq 0$
and thus $J(\gamma) \neq \{0\}$ as desired. The upper bound on $\lambda$
can also easily be read off, since 
$J_\gamma( \kkk^j, \kkk^k, \kkk^l) \equiv J(j,k,l) \kkk^{j+k+l+2\rho}
\bmod \p^{j+k+l+2\rho+1}$ and thus $\lambda \leq j+k+l+2\rho$. If two
of the elements in $\{j, k, l\}$ are equal, then $J(j,k,l) = 0$ by 
its definition and Lemma \ref{lem9}. Thus the maximal choice for 
$j,k,l$ with $J(j,k,l) \neq 0$ is $i+d-1, i+d-2, i+d-3$. Hence 
the upper bound $j+k+l+2\rho \leq 3i+3d-6 + 2\rho$ follows. 
\end{proof}

In the remainder of the section, we investigate the elements $J(j,k,l)$
with the aim to show that there is one non-trivial element among them
for $i \leq j, k, l < i+d$.

\begin{lemma}
\label{Jlemma}
Let $\gamma \in \hH_i$ and $i \leq j,k,l$. Write $d=p-1$.
\begin{items}
\item[\rm (a)]
$J(j, k, l)=J(l, j, k) \equiv -J(k, j, l) \bmod p$.
\item[\rm (b)]
$J(j, k, l)=J(j+d, k, l)=J(j, k+d, l)=J(j,k,l+d)$.
\item[\rm (c)]
$J(j, k, l) \equiv J(j+1, k, l)+J(j, k+1, l)+J(j,k,l+1) \bmod p$.
\end{items}
\end{lemma}

\begin{proof}
Parts (a) and (b) follow directly from Lemma \ref{lem9} and the
definition of $J$. It remains to prove (c). Lemma \ref{lem9} yields
\begin{eqnarray*}
a(j, k)a(j+k+\rho, l) 
  &=& a(j, k) a(j+k+1+\rho,  l)+a(j, k)b(j+k+\rho,  l+1)\\
  &=& (a(j+1, k)+a(j, k+1)) a(j+k+1+\rho,  l)\\
  && +a(j, k)a(j+k+\rho,  l+1).
\end{eqnarray*}
Replacing $(j,k,l)$ with $(k,l,j)$ and $(l,j,k)$ and summing up the 
resulting three term yields a proof for (c).
\end{proof}

%%%%%%%%%%%%%%%%%%%%%%%%%%%%%%%%%%%%%%%%%%%%%%%%%%%%%%%%%%%%%%%%%%%%%%%%%%%%%
\subsection{The proof of Theorem \ref{mt2}}
\label{pmt2}

We now proceed with the proof of the main Theorem \ref{mt2}. The lower
bound is already proved. It thus remains to show that $\lambda$ is finite.
We also determine an upper bound for it.

\begin{lemma}
\label{nonzero}
Let $\gamma \in \hH_i$. 
\begin{items}
\item[\rm (a)] 
There exists $j,k \in \{i, \ldots, i+d-1\}$ with $a(j,k) \neq 0$.
\item[\rm (b)]
If $a(j,k) \neq 0$, then either $a(j+1,k) \neq 0$ or $a(j,k+1) \neq 0$.
\item[\rm (c)]
There exists $j \in \{i, \ldots, i-d-1\}$ with $a(j,j+1) \neq 0$.
\end{items}
\end{lemma}

\begin{proof}
(a) This follows directly from the definition of $\rho(\gamma)$. \\
(b) This is obtained by Lemma \ref{lem9}(d). \\
(c) Suppose that all $a(j,j+1) = 0$. We show that this implies that
$a(j,k) = 0$ for all $j, k$ yielding a contradiction to (a). We assume 
that $j \leq k$ by Lemma \ref{lem9}(b)
and use induction on $k-j$. The assumpion asserts the claim for $k-j \leq 1$
and thus yields the initial step of the induction. Now assume that the claim
holds for all $i \leq j < k$ with $k-j \leq l$ for some $l > 1$. 
Then $0 = a(j,k) \equiv a(j+1,k) + a(j,k+1) \bmod p$ by Lemma \ref{lem9}(d). 
As $j+1-k \leq l$, the induction asserts $a(j+1,k) = 0$. Hence 
$a(j,k+1) = 0$ follows. This completes the inductive step.
\end{proof}

\begin{lemma}
\label{lastcor}
Let $\gamma \in \hH_i$. Write $c(k) = a(k+1,k)$ for $k \geq i$. If
$j \geq i$ and $t > 0$, then 
\[ a(j+t,j) \equiv \sum_{u=0}^{\lfloor \frac{t-1}{2}\rfloor}  
                         (-1)^{u}{{t-u-1}\choose{u}}c(j+u) \bmod p.\]
\end{lemma}

\begin{proof}
The result holds trivially for $t=1$. Let $t>1$ and denote $A_i=
\lfloor \frac{t-i}{2}\rfloor$. We use induction on $t$. Lemma \ref{lem9}
implies that 
\begin{eqnarray*}
   && a(j+t, j) \\
   &=& a(j+t-1, j)-a(j+t-1,  j+1)\\
   &=& \sum_{u=0}^{A_2}  (-1)^{u}{{t-u-2}\choose{u}}c(j+u)
        -\sum_{u=0}^{A_3}  (-1)^{u}{{t-u-3}\choose{u}}c(j+1+u)\\
   &=& c(j)+\sum_{u=1}^{A_2}(-1)^{u}{{t-u-2}\choose{u}}c(j+u)
        -\sum_{u=1}^{A_1}(-1)^{u-1}{{t-(u-1)-3}\choose{u-1}}c(j+u)\\
   &=& c(j)+\sum_{u=1}^{A_2}(-1)^{u}({{t-u-2}\choose{u}}
              +{{t-u-2}\choose{u-1}})c(j+u)\\
   && \hspace{1cm} +\delta (-1)^{\frac{\lfloor t-1\rfloor}{2}}
          {{t-A_1-2}\choose{A_3}}c(j+A_1),
\end{eqnarray*}
where $\delta=1$ if $A_1>A_2$  and $\delta=0$ if $A_1=A_2$.  
If $\delta=0$,   then the result follows from the last equation. It remains
to consider the case $\delta=1$. In this case $A_1=\frac{t-1}{2}$ and 
$A_3=\frac{t-3}{2}$. Hence $t-2-A_1=\frac{t-3}{2}$. Thus  
${{t-A_1-2}\choose{A_3}}={{\frac{t-3}{2}}\choose{\frac{t-3}{2}}}=1$.  
Also ${{t-A_1-1}\choose{A_1}}={{\frac{t-1}{2}}\choose{\frac{t-1}{2}}}=1$ 
and thus
\begin{eqnarray*}
&& \delta (-1)^{\frac{\lfloor t-1\rfloor}{2}}{{t-A_1-2}\choose{A_3}}c(j+A_1) \\
&=& (-1)^{A_1}c(j+A_1) \\
&=& (-1)^{A_1} {{t-A_1-1}\choose{A_1}}c(j+A_1). 
\end{eqnarray*}
Hence the result follows.
\end{proof}

{\bf Proof of Theorem \ref{mt2}.} \\
We now show that there exists integers $j,k,l \geq i$ such that $J(j,k,l) 
\neq 0$. Then Lemma \ref{Jlemma} implies that $J(j,k,l) \neq 0$ for some
$i \leq j, k, l < i+d$ and thus, in turn, Theorem \ref{bound} applies and
yields our desired result of Theorem \ref{mt2}.

For a contradiction we
assume that $J(j,k,l) = 0$ for all $j,k,l \geq i$. Let $m \ge i$ such that 
$m+\rho \equiv 1 \bmod d$. Lemma \ref{lem9} implies that
$c(j) = a(j+1,j) = a(j+2,j)$ and thus we obtain the following modulo $p$:
\begin{eqnarray*}
0 &=& J(j+1, j, m) \\
  &\equiv& a(j+1, j) a(j+1+j+\rho, m) \\
  && +a(j, m) a(j+m+\rho,  j+1) \\
  && +a(m, j+1) a(m+j+1+\rho,  j)\\
  &\equiv& c(j) a(j+1+j+\rho,  m) \\
  && +a(j, m) a(j+1,  j+1) \\
  && +a(m, j+1) a(j+2,  j)\\
  &\equiv& c(j) (a(2j+1+\rho,  m) -a(j+1, m)). \hspace{6.6cm} (1)
\end{eqnarray*} 
Choose $j$ so that $m < j < m+d$ and write $r = j-m$. Then 
$2j+1+\rho=2m+2r+1+\rho \equiv m+2r+2 \bmod d$. Further, using 
Lemma \ref{lastcor} yields that there exist $x(1), \ldots, 
x(r-1) \in \{0, \ldots, p-1\}$ with
\begin{eqnarray*}
&& a(2j+1+\rho, m)-a(j+1, m) \\
&=& a(m+2r+2,  m)-a(m+r+1, m) \\
&=& \sum_{u=0}^{\lfloor \frac{2r+1}{2}\rfloor} (-1)^{u}
       {{2r+2-u-1}\choose{u}}c(m+u)
    - \sum_{u=0}^{\lfloor\frac{r}{2} \rfloor}(-1)^{u}
        {{r+1-u-1}\choose{u}}c(m+u) \\
&=& (-1)^r {{2r+2-r-1}\choose{r}}c(m+r)
     + \sum_{u=1}^{r-1}x(u)c(m+u) \\
&=& (-1)^r(r+1)c(m+r)+\sum_{u=1}^{r-1}x(u)c(m+u). \hspace{6cm} (2)
\end{eqnarray*}
Combining Equations (1) and (2) yields
\[ 0=c(m+r)((-1)^r(r+1)c(m+r)+\sum_{u=1}^{r-1}x(u)c(m+u)). \hspace{4.6cm} (3)\]
Now we prove $c(m+r)=0$ for $0 < r \leq d-1$ by induction on $r$. For $r=1$
this follows directly from Equation (3). Let $r > 1$ and assume that $c(m+l)
= 0$ for $0 < l < r$. Then Equation (3) yields $c(m+r) = 0$. In summary,
$c(m+r) = 0$ for $0 < r \leq d-1$. Next, we show that $c(m) = 0$. For this,
note that $c(m)=-a(m, m+1)=-a(m+p-1, m+1)=-a(m+1+p-2, m+1)$. By Lemma
\ref{lastcor}, we obtain that $a(m+1+p-2, m+1)$ can be written as 
combination of $c(m+1), c(m+2), \ldots, c(m+\frac{p-3}{2})$. Hence $c(m)=0$
follows. In summary, we now proved $c(m)= \cdots=c(m+p-2)=0$. This implies
that $c(n)=0$ for all $n\ge i$ by using that 
$c(n)=a(n+1, n)=a(n+p, n+p-1)=c(n+p-1)$. 
Hence a contradiction to Lemma \ref{nonzero}(c) follows.

%%%%%%%%%%%%%%%%%%%%%%%%%%%%%%%%%%%%%%%%%%%%%%%%%%%%%%%%%%%%%%%%%%%%%%%%%%%%%
\subsection{The one-parameter case}
\label{pfmt3}

Throughout this section we assume that $p \geq 5$ is prime and $i \in 
\N_0$. This section considers the special case of $\gamma \in \hH_i$
with $\gamma = c \vartheta$ for some $c \in \U$ and $a\in 
\{2,\ldots,(p-1)/2\}$. 

Recall from Corollary \ref{R2} that $e_x$ is an eigenvector of $\sigma$ with 
$val(e_x) = x$ for each $x \in \Z$. Then $\sigma(e_x) = \omega^x e_x$. 
Also $\omega
\in \Z_p$ is a primitive $(p-1)$-th root of unity with $\omega \equiv
r \mod p$. Let $\sigma_a = \sigma^{k_a}$ and $\sigma_{1-a} = \sigma^{l_a}$
for certain $k_a, l_a \in \N$. 

\begin{definition}
\label{t} For $x,y,z\in \Z$ we define 
\begin{items}
\item[\rm (a)] 
$t_1(x, y, z)=c\sigma_a(c)(\omega^{2k_ax}\omega^{(k_a+l_a)y}\omega^{l_az}
                          -\omega^{(k_a+l_a)x}\omega^{2k_ay}\omega^{l_az})$,
\item[\rm (b)] 
$f(x, y, z) \in \{0, \ldots, p-1\}$ with 
$$t_1(x, y, z)+t_1(y, z, x)+t_1(z, x, y)=c^2a^{3i}(1-a)^{2i}f(x, y, z) 
    \bmod \p,$$
\item[\rm (c)] 
$t_2(x, y, z)=c\sigma_{1_a-a}(c)(
     -\omega^{(k_a+l_a)x}\omega^{2l_ay}\omega^{k_az}
     +\omega^{2l_ax}\omega^{(k_a+l_a)y}\omega^{k_az})$,
\item[\rm (d)] 
$g(x, y, z) \in \{0, \ldots, p-1\}$ with 
$$t_2(x, y, z)+t_2(y, z, x)+t_2(z, x, y)=c^2a^{2i}(1-a)^{3i}g(x, y, z)
     \bmod \p.$$ 
\end{items}
\end{definition}

\begin{lemma}
\label{L3}
\label{L4}
Let $x, y, z \in \Z$ and $\gamma = c \vartheta_a$.
\begin{items}
\item[\rm (a)]
$\gamma(\gamma(e_x \wedge e_y) \wedge e_z)
   =(t_1(x, y, z)+t_2(x, y, z))e_xe_ye_z$. 
\item[\rm (b)]
\begin{eqnarray*}
t_1(x, y, z)&\equiv& 
     c^2(a^{2x}(a(1-a))^y(1-a)^z-(a(1-a))^x(1-a)^{2y}(1-a)^z)\mod\p, \\
t_2(x, y, z)&\equiv&
     c^2(-(a(1-a))^x(1-a)^{2y}a^z+(1-a)^{2x}(a(1-a))^ya^z)\mod\p.
\end{eqnarray*}
\end{items}
\end{lemma}

\begin{proof}
(a) A straightforward calculation shows
\begin{eqnarray*}
&& \gamma(\gamma(e_x \wedge e_y) \wedge e_z) \\
&=& 
c\sigma_a(c)(\sigma_{a^2}(e_x)\, \sigma_{a(1{-}a)}(e_y)\,\sigma_{1{-}a}(e_z) 
- \sigma_{a(1{-}a)}(e_x)\, \sigma_{a^2}(e_y)\, \sigma_{1{-}a}(e_z)) \\
&& + 
c\sigma_{1-a}(c)(- \sigma_{a(1{-}a)}(e_x)\, \sigma_{(1{-}a)^2}(e_y)\, 
  \sigma_a(e_z)
+ \sigma_{(1{-}a)^2}(e_x)\, \sigma_{a(1{-}a)}(e_y)\, \sigma_a(e_z)).
\end{eqnarray*}
Corollary \ref{R2} shows that for each $u \in \N_0$ the element $e_u$ is 
an eigenvector of each of $\sigma_{a^2}$, $\sigma_{a(1-a)}$ and 
$\sigma_{(1-a)^2}$ with eigenvalues $\omega^{2k_au}, \omega^{(k_a+l_a)u}$ 
and $\omega^{l_au}$, respectively. It follows that
$\sigma_{a^2}(e_x)\sigma_{a(1-a)}(e_y)\sigma_{1-a}(e_z)
=(\omega^{2k_ax}\omega^{(k_a+l_a)y}\omega^{l_az})e_xe_ye_z$. 
Writing the other $3$ summands of $\gamma(\gamma(e_x \wedge e_y) \wedge e_z)$
in a similar way proves (a).\\
(b) 
We recall $\sigma_{a}=\sigma_r^{k_a}=\sigma_{r^{k_a}}$ and $\omega=r\mod \p$. Then it is easy to see from $p\O\le \p$ that $a=r^{k_a}\mod \p$, so $a=\omega^{k_a}\mod \p$, and similarly $1-a=\omega^{l_a}\mod \p$. The last part follows from applying $a=\omega^{k_a}\mod \p$, $1-a=\omega^{l_a}\mod \p$ in $t_1$ and $t_2$, and noting that $c=\sigma_a(c)=\sigma_{1-a}(c)$ modulo $\p$ by Lemma \ref{unitfixedgalois}.
\end{proof}

\begin{lemma}
\label{fandg}
Let $i \in \N_0$ and $x,y,z \geq i$. Write
$x=i+\bar{x}$, $y=i+\bar{y}$, and $z=i+\bar{z}$. Then 
\begin{align*}
   f(x,y,z)&= a^{2\bar{x}}(a(1-a))^{\bar{y}}(1-a)^{\bar{z}}-(a(1-a))^{\bar{x}}a^{2\bar{y}}(1-a)^{\bar{z}}+a^{2\bar{y}}(a(1-a))^{\bar{z}}(1-a)^{\bar{x}}\\
       &-(a(1-a))^{\bar{y}}a^{2\bar{z}}(1-a)^{\bar{x}}+a^{2\bar{z}}(a(1-a))^{\bar{x}}(1-a)^{\bar{y}}-(a(1-a))^{\bar{z}}a^{2\bar{x}}(1-a)^{\bar{y}}
\end{align*}
 and
\begin{align*}
    g(x,y,z)&=-(a(1-a))^{\bar{x}}(1-a)^{2\bar{y}}a^{\bar{z}}+(1-a)^{2\bar{x}}(a(1-a))^{\bar{y}}a^{\bar{z}}-(a(1-a))^{\bar{y}}(1-a)^{2\bar{z}}a^{\bar{x}}\\
        &+(1-a)^{2\bar{y}}(a(1-a))^{\bar{z}}a^{\bar{x}}-(a(1-a))^{\bar{z}}(1-a)^{2\bar{x}}a^{\bar{y}}+(1-a)^{2\bar{z}}(a(1-a))^{\bar{x}}a^{\bar{y}}.
\end{align*}
\end{lemma}

\begin{proof}
Using Lemma \ref{L4}, a technical calculation shows
\begin{align*}
    t_1(x,y,z) \bmod \p & =c^2(a^{2i+2\bar{x}}(a(1-a))^{i+\bar{y}}(1-a)^{i+\bar{z}}-(a(1-a))^{i+\bar{x}}a^{2i+2\bar{y}}(1-a)^{i+\bar{z}})\\
    =&c^2a^{3i}(1-a)^{2i}(a^{2\bar{x}}(a(1-a))^{\bar{y}}(1-a)^{\bar{z}}-(a(1-a))^{\bar{x}}a^{2\bar{y}}(1-a)^{\bar{z}}).
\end{align*}
Similarly we can get , 
    $$t_1(y,z,x)=c^2a^{3i}(1-a)^{2i}(a^{2\bar{y}}(a(1-a))^{\bar{z}}(1-a)^{\bar{x}}-(a(1-a))^{\bar{y}}a^{2\bar{z}}(1-a)^{\bar{x}})\mod \p,$$
$$t_1(z,x,y)=c^2a^{3i}(1-a)^{2i}(a^{2\bar{z}}(a(1-a))^{\bar{x}}(1-a)^{\bar{y}}-(a(1-a))^{\bar{z}}a^{2\bar{x}}(1-a)^{\bar{y}})\mod \p.$$

    Therefore, $t_1(x, y, z)+t_1(y, z, x)+t_1(z, x, y)=c^2a^{3i}(1-a)^{2i}f(x, y, z)\mod \p$. Similarly, we obtain $t_2(x, y, z)+t_2(y, z, x)+t_2(z, x, y)=c^2a^{2i}(1-a)^{3i}g(x, y, z)\mod \p$.
\end{proof}

\begin{lemma}
Let $i \in \N_0$ and $x,y,z \geq i$. Then 
$J_{\gamma}(e_x, e_y, e_z)\in \p^{x+y+z+1}$ if and only if 
\begin{equation}
a^{3i}(1-a)^{2i}f(x, y, z) + a^{2i}(1-a)^{3i}g(x, y, z) \equiv 0 \bmod{p}
\label{jc}
\tag{4}
\end{equation}
\end{lemma}

\begin{proof}
Using Lemma \ref{fandg} it follows that
 \begin{align*}
    &J_{\gamma}(e_x, e_y, e_z)\\
    =&(t_1(x, y, z)+t_2(x, y, z)+t_1(y, z, x)+t_2(y, z, x)+t_1(z, x, y)+t_2(z, x, y))e_xe_ye_z\\
    =&c^2(a^{3i}(1-a)^{2i}f(x, y, z)+a^{2i}(1-a)^{3i}g(x, y, z))e_xe_ye_z\mod \p^{x+y+z+1}.
\end{align*}
The last line of the above equation follows from the fact 
that $val(e_xe_ye_z)=x+y+z$. Hence, the result.
\end{proof}

{\bf Proof of Theorem \ref{mt3}.} \\
Theorem \ref{lboundlambda} shows $\lambda \geq 3i+3$. Since $|L_{i,\lambda}(\gamma)|=p^{\lambda-i}$, we have that
$|L_{i,\lambda}(\gamma)| = p^{2i+3+y}$
for some $y \in \N_0$. 
We aim to show that $y\le 2$. Suppose for a contradiction, $y\ge 3$. Then $\lambda \geq 3i+6$. Hence $J_{i}(\gamma) \subseteq \p^{3i+6}$.  Recall the functions $f$ and $g$ from Definition \ref{t}. Let 
\begin{eqnarray*}
F_1 &=& f(i,i+1,i+2), \\
G_1 &=& g(i,i+1,i+2), \;\; \mbox{ and }\\
E_1 &=& a^{3i}(1-a)^{2i}F_1+a^{2i}(1-a)^{3i}G_1.
\end{eqnarray*}
Then $E_1 \equiv 0 \bmod p$ by choosing $(x, y, z)=(i, i+1, i+2)$ in 
Equation \eqref{jc}.  Similarly,
\begin{eqnarray*}
F_2 &=& f(i,i+1,i+4), \\
G_2 &=& g(i,i+1,i+4), \;\; \mbox{ and } \\
E_2 &=& a^{3i}(1-a)^{2i}F_2+a^{2i}(1-a)^{3i}G_2.
\end{eqnarray*}
Thus $E_2 \equiv 0 \bmod p$ by choosing $(x, y, z)=(i, i+1, i+4)$ in 
Equation \eqref{jc}.
Further, we have that $F_2E_1-F_1E_2 \equiv 0 \bmod p$. On the other hand,
direct calculation yields that 
$F_2E_1-F_1E_2 = a^{2i}(1-a)^{3i}\big(F_2G_1-F_1G_2\big)$.
Hence $a^{2i}(1-a)^{3i}\big(F_2G_1-F_1G_2\big) \equiv 0 \bmod p$.
Now a direct calculation shows that
$F_2G_1-F_1G_2=a^4(1-a)^4(2a-1)^3(a^2+a-1)(a^2-3a+1)$, and we note
that $a, 1-a, 2a-1$ are all non-zero modulo $p$ since $a\in 
\{2, \ldots (p-1)/2\}$.  So 
\[ (a^2+a-1)(a^2-3a+1) \equiv 0 \bmod p.\]
A straightforward calculation using Lemma \ref{fandg} shows that 
\begin{align*}
 &F_1=f(i, i+1, i+2)=a(1-a)^2(2a-1)(a^2+a-1),\\
 &G_1=g(i, i+1, i+2)=a^2(1-a)(2a-1)(a^2-3a+1).
\end{align*}
Since $E_1=a^{3i}(1-a)^{2i}F_1+a^{2i}(1-a)^{3i}G_1 \equiv 0 \bmod p$, it 
follows that if $a^2-3a+1 \equiv 0 \bmod p$,  then $a^2+a-1 \equiv 0 
\bmod p$. Next, if $a^2-3a+1 \equiv 0 \bmod p$ then $a^2+a-1
=(a^2-3a+1)+2(2a-1)=2(2a-1) \ne 0 \bmod p$. Thus $a^2-3a+1\ne 0 \bmod p$. 
Similarly, $a^2+a-1\ne 0 \bmod p$. Therefore, 
$(a^2+a-1)(a^2-3a+1)\ne0 \bmod p$. This is a contradiction, 
hence $y\le 2$.

We consider the special case $p=5$. Here $a=2$ holds. Taking $a=2$ in 
$F_1$, we get that $f(i, i+1, i+2) = 0 \bmod 5$ and $g(i, i+1, i+2) \ne 0 
\bmod 5$. Hence $a^{3i}(1-a)^{2i}f(i, i+1, i+2) 
+ a^{2i}(1-a)^{3i}g(i, i+1, i+2) \ne 0 \bmod{p}$.
By Equation \eqref{jc}, it follows that $J_{\gamma}(e_i, e_{i+1}, e_{i+2}) 
\not \in \p^{3i+4}$.  Therefore $J(\gamma)\not \subseteq \p^{3i+4}$ and 
hence $\lambda=3i+3$. This implies that $y = 0$ in this special case.

\bibliographystyle{abbrv}

\end{document}